\documentclass[11 pt]{amsart}
\usepackage[showframe=false, margin = 1.35 in]{geometry}
\usepackage[shortlabels]{enumitem}
\usepackage{theoremref}
\usepackage{tikz}
\usepackage{subfigure, float}
\usepackage{pgfplots}
\usepackage{comment}
\usepackage{mathrsfs}

\newcommand{\T}{\mathcal{T}}
\newcommand{\ZZ}{\mathbb{Z}}
\newcommand{\OZ}[0]{\vec{\mathbb Z}}

\newcommand{\f}{\frac}
\newcommand{\ind}[1]{\mathbf{1}{\{ #1 \}}}

\newcommand{\0}{\mathbf 0}

\newcommand{\npa}{\nu_{||}}

\DeclareMathOperator{\CFM}{CFM}

\DeclareMathOperator{\Bin}{Bin}

\newtheorem{thm}{Theorem}
\newtheorem{lemma}[thm]{Lemma}

\theoremstyle{remark}

\theoremstyle{definition}

\usepackage{mathtools}

\mathtoolsset{showonlyrefs}

\title{Critical percolation and A+B$\to$2A dynamics}

\author{Matthew Junge}
\email{jungem@math.duke.edu}

\begin{document}
	\maketitle
	 
\begin{abstract}
We study an interacting particle system in which moving particles activate dormant particles linked by the components of critical bond percolation. Addressing a conjecture from Beckman, Dinan, Durrett, Huo, and Junge for a continuous variant, we prove that the process can reach infinity in finite time i.e., explode. In particular, we prove that explosions occur almost surely on regular trees as well as oriented and unoriented two-dimensional integer lattices with sufficiently many particles per site. The oriented case requires an additional hypothesis about the existence and value of a certain critical exponent. We further prove that  the process with one particle per site expands at a superlinear rate on integer lattices of any dimension.
Some arguments use connections to critical first passage percolation, including a new  result about the existence of an infinite path with finite passage time on the oriented two-dimensional lattice. 
\end{abstract}

\section{Introduction}


The goal of this work is to better understand expansion in 
 a growing system of random walks with $A+B \to 2A$ ``frog model" dynamics. In such systems $A$-particles, traditionally referred to as ``active frogs," are mobile. They perform simple random walk and activate stationary $B$-particles, referred to as ``sleeping frogs," upon contact. We introduce a variant in which sleeping frogs are linked by the geometry of critical bond percolation; a visit to a sleeping frog simultaneously wakes every sleeping frog in the connected cluster containing that frog. 
This makes it so large regions are activated instantly which dramatically increases the spread of activated particles. Although we keep with convention and continue to use the frog imagery, these dynamics more truly correspond to combustion; $A$-particles represent diffusing heat and $B$-particles entwined fuel cells (as in \cite{combustion, williams}).  

We first informally describe some related graph percolation models. \emph{First passage percolation} is the random metric induced by assigning independent passage times to each edge according to a distribution function $F$. \emph{Bond percolation} is the special case that edges are $p$-open (have passage time $0$) with probability $p$ and otherwise  are $p$-closed (the passage time is infinite). The critical value is the parameter $p_c$ above which an infinite component of $p$-open edges emerges. \emph{Critical first passage percolation} is the case in which $F(0) =p_c$.
%


 The \emph{critical frog model on $G$ with $m$ particles per site}, which we will often abbreviate as $\CFM(G,m)$, begins by placing $m$ sleeping frogs at each vertex. Call each component of vertices that are pathwise connected via $p_c$-open edges a \emph{cluster}. 
The process begins by waking up all frogs in the cluster containing the root. When a frog wakes up, it diffuses according to a simple random walk on $G$ with independent exponential(1)-distributed delays between jumps. Whenever a frog jumps on a cluster with sleeping frogs, every frog in the cluster wakes up and begins its own independent random walk. We will call every vertex in such a cluster \emph{activated}. 

Recently, Dinan, Durrett, Beckman, Huo, and Junge generalized the frog model to $d$-dimensional Euclidean space \cite{BMF}. They replaced simple random walks with Brownian motions and the underlying graph structure by distributing particles according to a Poisson point process and including a disk of radius $r$ around each point. Whenever an active particle comes within $r$ of a cluster of overlapping disks, all dormant particles in that cluster activate and begin their own Brownian motions. 

Since entire clusters of overlapping disks activate simultaneously, it is important to understand the cluster geometry. This is a special case of continuum percolation called the \emph{Boolean model}. It is known that there is a critical value $r_d$ such that for $r > r_d$ an infinite cluster of overlapping disks forms in $\mathbb R^d$ \cite{meester}. Once this cluster is discovered, the Brownian frog model has infinitely many moving particles. The main result of \cite{BMF} is that for $r< r_d$ the set of visited sites scaled linearly by time has the Euclidean ball with a deterministic radius (that depends on $r$) as its limiting shape.

The main open question from \cite{BMF} was to describe the rate the Brownian frog model expands at criticality $(r= r_d)$. The expansion is conjectured to be superlinear, but the authors did not try to guess the exact rate. 
The inspiration for the critical frog model comes from this conjecture. By moving to graphs, where percolation is better understood, we are able to prove concrete results. Continuum percolation in Euclidean space and bond percolation on the integer lattice are believed to have similar behavior at criticality. So, a result in either setting is suggestive of what occurs in the other. 

Similar continuous time $A+B\to 2A$ dynamics were studied by Ram\'irez and Sidoravicius  in \cite{combustion}. They viewed the active/asleep dynamics as a toy model for heat diffusing and igniting fuel. The motivation was to rigorously study a caricature  of chemical reactions associated to the steady-state burning of a homogeneous solid \cite[Chapter 9]{williams}. Note that they considered the process without clusters linking particles, so a visit to a vertex only activates the particles at that site.
%
Using different methods, Alves, Machado, and Popov in \cite{shape} concurrently proved a shape theorem similar to \cite{combustion}, but in discrete time. Subsequently, the discrete-time process has been extensively studied on lattices and trees (see \cite{frogs, new_drift,  HJJ1}  for a start). Note that bond percolation plays an important role in \cite{dobler2018recurrence}. 

Another closely related $A+B \to 2A$ model was studied in a sequence of papers by Sidoravicius and Kesten \cite{KS1, KS2, KS3}. With the intention of modeling the spread of a rumor or infection, they considered the process in which both $A$- and $B$-particles perform continuous time simple random walks. Starting with one $A$-particle at the origin, they proved that the set of locations of $A$-particles expands linearly in time. One could define an analogous critical process with these dynamics; when an $A$-particle meets a $B$-particle, it and all other $B$-particles that started in the same cluster convert to $A$-particles.  It would be interesting, and perhaps more tractable, to prove that this process explodes on lattices and trees.


\subsection{Results}
 
Let $G = (\mathcal V, \mathcal E)$ be a connected, infinite graph with root $\0$. The examples we focus on are: 


\begin{itemize} 
	\item The $d$-ary tree $\mathbb T_d$ which is the rooted tree in which each vertex has $d$ children. 
	\item The integer lattice $\mathbb Z^d$ has the standard basis $e_i = (0,0,\hdots,0,1,0,\hdots,0)$ for $1 \leq i \leq d$. The set of vertices are all integer linear combinations of $(e_i)$, and edges exist between any vertices separated by distance $1$. 
	\item The oriented lattice $\OZ^d$ which is the usual integer lattice in which each edge $(v, v+ e_i)$ is directed towards $v+e_i$. 
\end{itemize}

 Let $\xi_t$ be the set of activated vertices at time $t$. We say that $\CFM(G,m)$ is \emph{explosive} if there exists an almost surely finite time $T$ such that $|\xi_T| = \infty.$ Note that the possibility of an explosion requires that the process evolves in continuous time. We use a comparison to critical first passage percolation to show that the critical frog model is explosive with enough particles per site. The statement for the oriented case relies on an assumption about the critical exponent associated with the correlation length in the temporal direction, which is commonly denoted by $\npa$ and is defined at \eqref{eq:npa}.
 
 \begin{thm} \thlabel{cor:fpp} 
$\CFM(G,m)$ is explosive for:
\begin{enumerate}[label = (\roman*)]
	\item \label{Z} $\mathbb Z^2$ with $m=4$,
	\item \label{OZ} $\OZ^2$ with $m=2$ assuming $\npa$ defined at \eqref{eq:npa} exists and is strictly less than $2$, and 
	\item \label{tree} $\mathbb T_d$ with $m=1$.
\end{enumerate}
\end{thm}

 The result for $\mathbb Z^2$ follows from an extension of \cite{damron2017} to the case $F(+ \infty) <1$.  We prove a new result, \thref{thm:OZ}, for $\OZ^2$ to establish \ref{OZ}. 
As for our hypothesis on $\npa$, no rigorous proof exists, but it is widely believed that $\npa$ exists and is approximately equal to $1.74$ \cite{npa}. 
 To prove \ref{tree}, we work with an embedded process restricted to the leaves of percolation clusters and apply a result about explosive branching processes with heavy tails  from \cite{explosion}.

Note that by a straightforward coupling, the results of \thref{cor:fpp} are also true for larger values of $m$. We conjecture that the critical frog model is explosive with $m=1$ on $\mathbb Z^d$ and $\OZ^d$ for $d\geq 2$, but prove a weaker result.
Let $M_t = \max \{\|v- \0\| \colon  v \in \xi_t\}$ be the maximum distance between sites of $\xi_t$ and the root. 

\begin{thm} \thlabel{thm:superlinear} Consider $\CFM(G,1)$. Let $\epsilon>0$ be arbitrary and $d \geq 2$. The following hold almost surely:
\begin{enumerate}[label = (\roman*)]
\item \label{oriented} If $G=\OZ^2$, then $M_t/t^{5-\epsilon} \to \infty$. 
\item \label{unoriented} If $G = \mathbb Z^d$ or $\OZ^d$, then $M_t / t \to \infty$.
\end{enumerate}
\end{thm}

The proof for the oriented case goes by  tracking the maximally displaced point in one ancestry line. A lower bound for box-crossing probabilities from \cite{box} guarantees that the jump size in some direction is at least $n$ with probability $ct^{-1/5}$ for some fixed $c>0$.  Since we only consider one particle, we wait an $\text{exponential}(1)$-distributed time between each jump. The law of large numbers then gives the claimed growth of this ancestry line, which serves as a lower bound for $M_t$. For higher dimensions we do not have as good of a tail estimate, but it is not so hard to prove that the expected jump size is infinite at criticality, which is enough to deduce $M_t/t \to \infty$.

The unoriented setting has two extra complications. First, we use a subprocess that jumps when a new half-space with a sufficiently large cluster is reached in the $e_d$-direction. A similar argument as for the oriented case shows that the cluster size has infinite expectation. This requires an additional observation from \cite{uhp} that $p_c$ is the same on $\mathbb Z^d$ and on the $d$-dimensional half-space. The second difficulty is that, unlike in the oriented case, frogs might move back towards the origin. Thus, we need to control the number of active frogs near the front of the subprocess in order to ensure that the time between jumps has finite expectation. This relies on a relatively simple estimate for the hitting time of a set for multiple one-dimensional simple random walks.

\subsection{Critical first passage percolation} \label{sec:fpp_def}


Assign independent uniform$(0,1)$ random variables $(\omega_e)_{e \in \mathcal E}$ to each  $e \in \mathcal E$. In \emph{$p$-bond percolation}, $e\in \mathcal E$ is \emph{$p$-open} if $\omega_e \leq p$. Otherwise, it is \emph{$p$-closed}. The \emph{critical value} is the infimum over $p$ such that there is almost surely an infinite path of adjacent $p$-open edges:
$$p_c= p_c(G) = \inf \{ p \colon P(\text{there exists an infinite $p$-open path}) =1\}.$$
It is easy to show that $p_c(\mathbb T_d) = 1/d$ by comparison to a critical branching process. In \cite{kesten1/2}, Kesten proved that $p_c(\mathbb Z^2) =1/2$. The exact value of $p_c(\mathbb Z^d)$ is not known in higher dimensions, nor is $p_c(\OZ^d)$ known for any $d$. 
%

Let $F$ be a distribution function with $F(x) =0$ for $x<0$ with $F(+\infty) =1$. The inverse of $F$ is defined as
$$F^{-1}(y) =  \inf \{ y \colon F(y ) \geq x\}.$$
  Define the \emph{passage time along edge $e$} as $t_e = F^{-1}(\omega_e)$. Assign to each \emph{vertex self-avoiding path}, $\gamma = (x,v_1,v_2,\hdots, v_n, y)$ with all vertices distinct, the passage time $T(\gamma) = \textstyle \sum_{(u,v) \in \gamma} t_{(u,v)},$
and define the \emph{passage time from $x$ to $y$} as
$$T(x,y) = \inf \{ T(\gamma) \colon \gamma \text{ is a path from $x$ to $y$} \}.$$
The random metric space induced by $T(\cdot\:, \cdot)$ is referred to as \emph{first passage percolation}. 
Define the \emph{passage time to infinity} as the limit of the minimal passage times to distance $n$:
\begin{align}
\rho=\rho(G,F) = \lim_{n \to\infty} \inf\{ T(\0, x)\colon \|x\| = n\}.\label{eq:rho}
\end{align}
Note that the limit exists for trees and lattices by monotonicity. The Kolmogorov $0$-$1$-law ensures that $P(\rho < \infty) \in \{0,1\}$.

Considerable attention has been devoted to the geometry of $p_c$-bond percolation. On trees, a simple comparison to a critical branching process shows that there are no infinite $p_c$-open paths. The same is believed to hold for  $\mathbb Z^d$. This is only known in dimension $2$ \cite{kesten1/2} and in high-dimensions \cite{slade}. Critical first passage percolation is a natural way to relax $p_c$-bond percolation.

 Zhang initially studied $F$ with $F(0) > p_c$ in \cite{zhang0}. It follows immediately that $\rho < \infty$ for such $F$, since there is an infinite zero-passage time path.  On the other hand, for $F(0)<p_c$ it is not hard to show that $\rho= \infty$ almost surely \cite[Proposition 4.4]{fpp}. For these reasons,  first passage percolation with any $F$ satisfying $F(0) = p_c$ is referred to as \emph{critical first passage percolation}. 
 
 In a followup work \cite{zhang}, Zhang investigated the critical case and proved that a ``double behavior" occurs on $\mathbb Z^2$. In particular, he showed that for 
\begin{align}F_a(x) = \begin{cases} 0, & x < 0  \\ p_c + x^{a}, & p_c+x^{a}\leq 1 \\ 1 , &p_c + x^{a} > 1  \end{cases}\label{eq:F_a}
\end{align}
it holds that $\rho(\mathbb Z^2, F_a) < \infty$ almost surely for small enough $a$. On the other hand, for distributions that approach $p_c$ more rapidly,
\begin{align}G_b(x) = \begin{cases} 0, & x < 0  \\ p_c + \exp(-1/x^{b} ), &  p_c + \exp(-1/x^{b} ) \leq 1 \\ 1, &  p_c + \exp(-1/x^{b} )  >1 \end{cases},\label{eq:G_b}
\end{align}
it holds that $\rho(\mathbb Z^2, G_b) = \infty$ for large enough $b$. 
Rather recently, Damron, Lam, and Wang in \cite{damron2017} gave exact conditions on $F$ with $F(+\infty)=1$ that lead to finite $\rho$: 
\begin{align}
\rho(\mathbb Z^2, F) < \infty \text{ if and only if } \sum_{k=1}^\infty F^{-1}(p_c + 2^{-k}) < \infty.\label{eq:damron_iff}
\end{align}

Their approach relies on the relative wealth of knowledge about percolation in two dimensions. It appears to be difficult to prove analogues of \eqref{eq:damron_iff} for $\mathbb Z^d$ with $d >2$. In fact, proving that $\rho$ is finite in the case of bond-percolation on $\mathbb Z^3$ is one of the major open problems in the field. However, a result of Bramson for branching random walk that predates the introduction of critical first passage percolation tells us what happens on trees. A direct consequence of \cite[Theorem 2]{Bramson1978} is that 
\begin{align}\text{ $\rho (\mathbb T_d , F) < \infty$ a.s.\ if and only if $\exists \lambda >0$ with  }\sum_{k=1}^\infty F^{-1}\left(p_c + \exp(-\lambda^{k})\right) < \infty.\label{eq:tree_iff}\end{align}
When $\rho = \infty$, Bramson's theorem additionally describes the growth rate: 
\begin{align}T_n(\mathbb T_d, F) \approx \sum_{k=1}^{\log \log n} F^{-1}(p_c + \exp(2^{-k})).\label{eq:tree_rate}\end{align}
It seems to us that this reformulation of Bramson's work in terms critical first passage has not been observed in the literature. This is especially interesting because \eqref{eq:tree_iff} is markedly weaker than the recently proven criterium at \eqref{eq:damron_iff}. For instance, $\rho(\mathbb T_d, G_b) < \infty$, but $\rho(\mathbb Z^2, G_b) = \infty$ for all $b>0$. Also, \eqref{eq:tree_rate} implies that when $F$ is such that $p_c$-closed edges have passage time 1 rather than infinity, we have $T_n(\mathbb T_d, F) \approx \log \log n$. This contrasts with the result of Chayes, Chayes, and Durrett that $T_n(\mathbb Z^2,F) \approx \log n$ in this setting \cite{ccd}.  

Obtaining a similar necessary and sufficient condition as either \eqref{eq:damron_iff} or \eqref{eq:tree_iff} for $\OZ^2$ appears challenging. The main difficulty is that we do not understand $p$-bond percolation with $p$ near $p_c$ nearly as well as with the unoriented case. We instead prove a  result tailored for the critical frog model in which we assume that a certain critical exponent exists. We delay defining it until just before the proof. Below we write $F(x) = p_c + O(x^a)$ to mean $0 \leq F(x)-p_c \leq Cx^a$ for some $C>0$ and all $0<x < \epsilon$ for some small $\epsilon >0$. 

\begin{thm} \thlabel{thm:OZ} Suppose that the critical exponent $\npa$ defined at \eqref{eq:npa} exists. If $F(x)= p_c + O(x^a)$ with $a^{-1} > \npa -1$, then $\rho < \infty$ a.s. 
\end{thm}



Our proof of \thref{thm:OZ} uses the block construction for oriented percolation devised by Durrett in \cite{RD84}. The modified proof is inspired by what was recently done for an inhomogeneous percolation model studied by Cristali, Junge, and Durrett \cite{Pperc}. The difference is that they worked with $p$-bond percolation for $p = p_c+\epsilon +o(1)$ for fixed $\epsilon>0$. In this work, we must consider $p = p_c + o(1)$. This makes critical exponents (i.e., the behavior of $p$-bond percolation very near $p_c$) a relevant quantity.

The idea is to map a lattice of interlocking parallelograms to a copy of $\OZ^2$. If one such parallelogram has a ``centered crossing," then we declare the corresponding bond open in the renormalized lattice. The sizes and slopes of the parallelograms depend on their distance to $\0$ and on $\npa$. Making them suitably larger than the correlation length causes one-dependent bond percolation on the renormalized process to be supercritical and have an infinite connected path. We show that this path has a passage time that is summable so long as $a$ is sufficiently small.

\subsection{Organization}
Section \ref{sec:comparison} uses results from critical first passage percolation to prove \thref{cor:fpp}. Section \ref{sec:m=1} contains the the proof of \thref{thm:superlinear}. Section \ref{sec:OZ} is dedicated to \thref{thm:OZ}.

\section{Explosions}
\label{sec:comparison}

The following arguments rely on a connection between critical first passage percolation and the critical frog model. Let $$B_t(G,F) = \{ x \in \mathcal V \colon T(\0,x) \leq t\}$$ be the set of sites that can be reached from $\0$ by time $t$ in first passage percolation. To prove the first two parts of \thref{cor:fpp}, we will show that $\CFM(G,m)$ can be coupled to critical first passage percolation so that $B_t(G,F) \subseteq \xi_t$ for an appropriate choice of $F$. A technical difficulty is that $F(+\infty) <\infty$, so there are infinite passage time edges. This case is not directly covered by \eqref{eq:damron_iff} or \thref{thm:OZ}. We explain in the proofs how the results are extended to this setting.

\begin{proof}[Proof of \thref{cor:fpp} \ref{Z}]
	We claim that the set of activated sites in $\CFM(\mathbb Z^2, 4)$ dominates $B_t(\mathbb Z^2, F)$ with passage times specified by
\begin{align} F(x) =  p_c + \f{(1-p_c)(1 - e^{-x})^2}{16}
.\label{eq:te}	
\end{align}
Equivalently, $$t_e = \begin{cases} 0, & p_c \\ \max(X_e^{(1)},X_e^{(2)}), & (1-p_c)/16 \\ \infty, & \text{otherwise} 	
 \end{cases}$$
with $X_e^{(1)}$ and $X_e^{(2)}$ independent unit exponential random variables. For the remainder of this proof we let $F$ and $t_e$ refer to this distribution and random variable.

The coupling between the $\CFM(\mathbb Z^2, 4)$  and $F$-critical first passage percolation is fairly straightforward to describe. At each vertex $x$ assign the edges attached to $x$ in a one-to-one manner to the four sleeping frogs at $x$. The frogs, when woken up, still move to a random neighboring vertex, but we will only use that frog if it moves across its pre-assigned edge and that edge was not already $p_c$-open.

More precisely, we give an edge $e=(x,y)$ passage time $0$ if it is open in the underlying $p_c$-bond percolation. If $e$ is not open, but both frogs at $x$ and $y$ assigned to $e$ move across it in their first step, then we assign the maximum of their two passage times to $e$. Otherwise $e$ has passage time $
\infty$. The set of activated sites in this subprocess is dominated by the critical frog model in which frogs are only allowed to move one step along their preassigned edge. This is because $t_e$ is 0 with the same probability in each process. Alternatively, if $t_e>0$, it is stochastically larger than the time it takes for a frog to cross $e$. This is because we take the maximum of the two crossing times of the frogs at each end of $e$.

The criterium at \eqref{eq:damron_iff} does not immediately apply when infinite passage times are possible. However, it does apply to $F^{(1)}$ induced by the truncated passage times $t_e^{(1)} = t_e \wedge 1$. It is straightforward to check that $F^{(1)}(x) -p_c = O(x^{-2})$ as $x \to 0$, and so the criterium at \eqref{eq:damron_iff} to have $\rho(\mathbb Z^2, F^{(1)}) < \infty$ a.s.\ is satisfied. It is proven in \cite[Lemma 4.3]{fpp} that $\rho < \infty$ implies there exists an infinite vertex self-avoiding path $\gamma$ from $\0$ to infinity. Since the passage time is finite, such a path uses only finitely many weight-$1$ edges. So, all edges beyond some almost surely finite distance $R$ have passage time less than $1$. Using the fact that $t_e$ and $t^{(1)}_e$ are coupled via $\omega_e$, it follows that critical first passage percolation with $F$ has an infinite vertex self-avoiding path outside of the ball of radius $R$ with finite passage time.
The time to discover this path is almost surely finite because the shape theorem for the frog model from \cite{combustion} ensures that set of activated sites in the usual frog model (one particle per site, no identification among percolation clusters) a.s.\ contains a linearly expanding ball. Since $\CFM(\mathbb Z^2,4)$ dominates the usual frog model, it follows that $\CFM(\mathbb Z^2,4)$ is explosive .  
\end{proof}

The proof for the oriented lattice is similar.

\begin{proof}[Proof of \thref{cor:fpp} \ref{OZ}]
	The analogous distribution to \eqref{eq:te} to compare to $\CFM(\OZ^2,2)$ is
\begin{align}
F(x) = p_c + (1-p_c)(1-e^{-x})/2\label{eq:teo}
\end{align}
so that 
$$t_e = \begin{cases} 0, & p_c \\ X_e , & (1-p_c)/2 \\ \infty, & \text{otherwise} 	
 \end{cases}$$
 with $X_e$ a unit exponential random variable.
The function and coupling are a little cleaner here since frogs must move away from $\0$ and there are only two outward edges at each site. Once again, we only allow frogs to move one step and along a predetermined outward edge assigned in a bijective manner. If the frog moves along its assigned edge, then the passage time is $X$-distributed. If the frog does not move along that edge, then the passage time is infinite. 

It is easy to check that the truncated passage times $t^{(1)}_e = t_e \wedge 1$ satisfy the hypotheses of \thref{thm:OZ} with $a=1$. Thus, assuming the other hypothesis $\npa < 2$ (defined in Section \ref{sec:OZ}) holds, we have $\rho(\OZ^2, F^{(1)}) < \infty$ a.s.\ with $F^{(1)}$ the distribution induced by $t^{(1)}_e$. As this path can contain only finitely many edges with weight-$1$, it follows that $\rho(\OZ^2,F) < \infty$ with positive probability. To translate this into an almost sure statement about the frog model, notice that each time $M_t$ increases for the frog model, a new embedded subprocess can be started from that point because the edges beyond have yet to be discovered. With positive probability the embedded first passage percolation model reaches infinity in finite time, and thus the frog model is explosive. Since this happens infinitely many times, the probability of an explosion is $1$.
\end{proof}

The idea with the tree is similar, but now we can use independence among subtrees to construct a subprocess that follows a growing number of lineages.

\begin{proof}[Proof of \thref{cor:fpp} \ref{tree}]
We can use the edge weights $\omega_e$ and information about the movement of frogs to induce a Galton Watson tree $\T$ whose vertices are a subset of $\mathbb T_d$. A \emph{leaf} in a bond percolation cluster is a vertex whose parent edge is open, but all child edges are closed.  Let $L(v)$ be the set of all leaves in the $p_c$-bond percolation cluster started at vertex $v$. For each $u\in L(v)$ let $A_u$ be the event that the first step of the frog at $u$ is to a child vertex of $u$. Set $C_v = \{ u \in L(v) \colon \ind{A_u} = 1\}$. Notice that 
$$|C_v| = \text{binomial}(|L(v)|, d/ (d+1)).$$
 
  The tree $\T$ starts with root $\0$. The first level of $\T$ consists of all vertices from $C(L(\0))$. The vertices attached to each $v$ at level $i+1$ are  the vertices from $C(L(v))$. To each edge $(u,v)$ in $\T$, we attach an independent unit exponential passage time. 
Using independence, and self-similarity of $\mathbb T_d$, it follows that $\T$ is a Galton-Watson tree with offspring distribution $Z = |C_{\0}|$. Another advantage of $\T$ is that it tracks the edges with non-zero passage time in the frog model in which we only consider frogs at the leaves of critical bond percolation clusters that jump away from the root on their first step. After this jump these frogs are removed.

Let $Z'$ be the size of the cluster containing the root in critical percolation on $\mathbb T_d$. It is a classical result that $$P(Z' \geq n^2) = \f{c'}{n}(1 + o(1))$$ for some $c'>0$ (see \cite[Theorem 2.1]{tree_critical}). It is proven in  \cite[Theorem 2]{minami} that the number of leaves in a Galton-Watson tree, call it $Z''$, corresponds to a Galton-Watson process with the same asymptotic right-tail behavior as $Z'$. Thus, 
\begin{align} P(Z'' \geq n^2)  \geq \f {c''} n
\end{align}
for all $n$ and some $c''>0$. Recall that our offspring distribution of interest is $Z = |C_\0|$. A binomial thinning of an infinite random variable does not change the right-tail behavior, so we have
\begin{align} P(Z \geq n^2)  \geq \f {c} n\label{eq:Ztail}
\end{align}
for all $n$ and some $c>0$.
 The above line ensures that the offspring distribution $Z$ satisfies the ``plumpness" criterium of \cite[Theorem 1.3]{explosion}. The condition to check for an explosion to occur is that $F^{-1}( e^{ - \lambda ^k})$ is summable for some $\lambda>0$. As we are assigning independent unit exponential edge weights, it is straightforward to verify that this is summable for any $\lambda>0$. Thus, conditional on $\mathcal T$ being infinite, it almost surely contains an infinite path with finite passage time. 

As noted during the construction of $\mathcal T$, the set of sites visited by time $t$ in on $\mathcal T$ is equivalent to the set of activated sites in a modified frog model in which we remove all frogs at non-leaf vertices of critical percolation clusters and only allow the remaining frogs to take one step. It follows that this process reaches infinity in finite time with positive probability. To turn this into an almost sure statement consider a slightly modified process. The frog started at the root will almost surely escape to $\infty$ along a unique, uniformly sampled vertex self-avoiding path. Each vertex $v$ along this path contains a frog and also a copy of $\mathbb T_d$ rooted at the child vertex of $v$. With positive probability each of these subtrees will have infinitely many frogs activated in finite time. It follows that after an almost surely finite time one of these subtrees will be discovered by the frog at the root, and thus in finite time infinitely frogs will be activated.
\end{proof}

\section{Results with one particle per site}
\label{sec:m=1}

\subsection{The oriented lattice}

\begin{proof}[Proof of \thref{thm:superlinear} \ref{oriented}]
	
Analogous to \thref{cor:fpp} \ref{tree}, we will again use an embedded subprocess. Let $\mathcal C(v)$ denote the cluster containing the vertex $v$. Let $x_1$ be the point in $\mathcal C(\mathbf 0)$ with the largest distance from $\mathbf  0$. If there are multiple candidates, then choose one arbitrarily. Let $\tau_1$ be the time for the frog started at $x_1$ to jump to a neighboring vertex $y_1$. We then iteratively define $x_{i+1}$ to be the vertex of $\mathcal C(y_i)$ with maximal distance from $\0$, breaking ties as before. We set $y_{i+1}$ to be the vertex that the frog started at $x_{i+1}$ first jumps to after waiting $\tau_{i+1}$ time units.

Since we are on the oriented lattice and we are choosing points at maximal distance, we have the jump sizes $J_i= \|y_{i} - x_{i-1}\|$ are independent and identically distributed.
So, we can decompose $\|y_n\| = \sum_{i=1}^n J_i$. The lower bound from \cite[Theorem 1.1]{box} gives $P(J \geq n) \geq c n^{-1/5}$
for some $c >0$ and all $n\geq 1$. As the expected time between each jump has mean $1$, it follows by applying \cite[Theorem 3.8.2]{durrett} to  the embedded jump chain in $(M_t)$ that almost surely $$P(t^{-5+\epsilon} M_t \to \infty) =   P(n^{-5+\epsilon} M_n \to \infty) = 1.$$
The proof in higher dimension is similar, however we only know that $E J = \infty$, in which case we invoke the strong law of large numbers to deduce that $M_t/t \to \infty$. 

One way to show that $EJ = \infty$ goes by contradiction. Suppose that $EJ < \infty$. Then, if $J_h$ is the expected number of sites in the cluster at height $h$,  monotonicity ensures there is a value $h_0$ such that $E J_{h_0} < 1$. The total size of the cluster containing the origin is dominated by starting independent oriented percolation clusters of height $h_0$ at each site counted by $J_{h_0}$. The dominance holds because it can be thought of as adding additional edges to the graph. It follows that the cluster containing the origin is dominated by a subcritical branching process and is almost surely finite. 

Since there are only finitely many accessible edges from the origin up to height $h_0$, the quantity $E J_{h_0}$ is a polynomial and thus continuous. It follows that there is a value $p_c' > p_c$ such that $E_{p_c'} J_{h_0} <1.$ The same reasoning as in the previous paragraph implies that the cluster containing the origin in $p_c'$-bond percolation is almost surely finite. This contradicts the definition of $p_c$. Thus $E J = \infty$. 
\end{proof}

\subsection{The unoriented lattice}

We begin by giving an overview of the proof before turning to the details in the next paragraph. The idea with the unoriented lattice is to use that the cluster size when the process crosses a new half-space in the $e_d$-direction is independent if restricted to sites in the half-space. An issue is that the expected time for a single frog to cross into a new region is infinite. However, an elementary result for one-dimensional simple random walk tells us that the expected time is finite for three or more particles. Accordingly, we wait some amount of time to have three frogs. Next, we let them diffuse until they discover a cluster of size at least $3$. We then use the three nearest particles to the boundary to discover the next cluster of size $\geq 3$. The jumps at each new discovery have infinite expectation, and the distance between large enough discoveries, though random, is controllable. This yields an embedded process with superlinear growth. 

We start by stating and proving the lemma for simple random walk. In essence, it says that at least one of three independent random walks started at $-1,-2$ and $-3$ will reach a site a geometrically distributed distance away in finite expected time. 

\begin{lemma} \thlabel{lem:SRW}
Let $S_t^{(1)}, S_t^{(2)}, S_t^{(3)}$ be three simple random walks on $\mathbb Z$ with $S_0^{(i)} = -i$ for $i=1,2,3$. Let $Y$ be a geometric random variable supported on the nonnegative integers so that $P(Y= k) = (1-q)^kq$ for $k \geq 0$. Let $\sigma = \inf \{t \colon S_t^{(i)} = 2Y \text{ for some $i$}\}$ be the number of steps to first reach $2Y$. It holds that $E \sigma < \infty$.
\end{lemma}

\begin{proof}
If $T$ is the first time that a simple random walk started at $0$ reaches $1$, then it is a well known fact (see \cite{durrett}[Chapter 4]) that $P(T > t) \leq C t^{-1/2}$ for some $C>0$ and all $t >0$. If $T_k$ is the number of steps to reach $k$, then we can write $T_k$ as a sum of $k$ i.i.d.\ copies of $T$. The event $\{T_k > t\}$ is contained in the event that one of these copies is larger than $t/k$. Applying a union bound gives $$P(T_k> t) \leq k P(T > t/k) \leq C k(t/k)^{-1/2} = C k^{3/2} t^{-1/2}.$$ Conditional on the event $\{Y=k\}$, it is easy to see that, $$P(\sigma \geq t \mid Y=k) \leq C^3 (2k)^{9/2} t^{-3/2}.$$ Letting $C_0 = 2^{9/2}C^3$, it follows that 
$$E \sigma = \sum_{t=1}^\infty P(\sigma \geq 	t) \leq C_0 \sum_{k=1} ^\infty  k^{9/2}P(Y=k) \sum_{t=1}^\infty t^{-3/2} < \infty.$$
The above line is finite because $P(Y=k)$ decays exponentially fast, and $t^{-3/2}$ is summable.
\end{proof}

\begin{proof}[Proof of \thref{thm:superlinear} \ref{unoriented}]

Let $H$ be the expected vertical displacement of the cluster containing $\0$ in critical bond percolation on the upper-half space $\mathbb H^d = \{(x_1,\hdots, x_d) \colon x_d \geq 0\}$. The main theorem of \cite{uhp} is that the critical value for bond percolation on the half-space is equal to $p_c(\mathbb Z^d)$. A similar argument as at the end of the proof of \thref{thm:superlinear} \ref{oriented} then shows that $E H = \infty$ for all $d \geq 2$. We consider a restricted process which we now define.

 Initially, only the frog started at $0$ is allowed to diffuse. We wait until it visits three distinct sites which happens after a random time with exponential tail. Let $\Gamma_1$ be the union of the clusters containing this initial set of $3$ vertices.  Now we define an iterative exploration process.

Let $Z_1$ be the point with largest $x_d$ coordinate in $\Gamma_1$. 
 As $|\Gamma_1| \geq 3$, choose three sleeping frogs in $\Gamma_1$ that are nearest to the half-space $\mathcal P_1 = \{ x \colon x_d > Z_1\}$. If there are several candidates, then break ties arbitrarily. Wake these three frogs up and for all future steps we will disregard all other frogs (active or asleep) in $\mathcal P_1^c$.  

The displacement of each frog in the $e_d$ direction is a lazy simple random walk that stays in place with probability $(d-1)/d$. Each time the three frogs reach a new maximal $x_2$ coordinate in their collective range, a new cluster in the half-space at that site is discovered. Since we are working on the half-space, the geometry of the newly discovered cluster is independent of the past. After a geometric number of trials with a positive probability of success, a cluster of size at least 3 will be discovered. Each failure requires the frogs to traverse at most an additional 2 units in the $e_d$ direction to find another independent cluster. By \thref{lem:SRW}, the time to discover a large enough cluster is dominated by $\sigma$, which has finite expectation, say, $\mu<\infty$. 

Let $\Gamma_2$ be the newly discovered cluster with size at least $3$. We repeat the procedure from before. That is, we select three particles nearest the undiscovered boundary half-space, then disregard all frogs, except for the three newly activated ones, outside of that plane. These three frogs will discover a new large cluster after at most $\mu$ steps in expectation. Iterate this procedure indefinitely.

As mentioned in the first paragraph of this argument, the expected displacement is infinite at each step (conditioning the cluster to have size at least 3 only increases this quantity). Moreover, the time between each jump has expectation uniformly bounded by $\mu$. Letting $R_t$ be the maximal $x_2$ coordinate among the explored sites in this process, it follows from the strong law of large numbers that $R_t/t \to \infty$ almost surely. Since $M_t \geq R_t$, the claimed result follows.
\end{proof}

\section{Critical first passage percolation on the oriented lattice}
\label{sec:OZ}

 It is often convenient to represent $\OZ^2$ with a rotated and rescaled lattice ${\mathcal L} = \{ (m,n) \in \ZZ^2 \colon m+n \hbox{ is even} \}$
with oriented edges from $(m,n) \to (m+1,n+1)$ and $(m,n) \to (m-1,n+1)$. 
 We start by defining the critical exponent from the statement of \thref{thm:OZ}.

As mentioned earlier, there is a value $p_c(\OZ^2) \in (0,1)$. In what follows, we work on the rotated lattice $\mathcal L$ defined at the beginning of Section \ref{sec:fpp_def}. Let $r_n$ be the rightmost point at height $n$ that can be reached by a $p$-open path started from a vertex on the negative half-line $(-\infty, 0] \times \{0\}$. It is known that $\alpha(p) = \lim_{n \to\ infty} r_n /n$ exists, is positive for $p > p_c$ \cite{RD84} and that $\alpha(p_c) = 0$. Supercritical percolation clusters grow like a cone with slope $1/\alpha(p)$. Looking towards defining the correlation length, let $\sigma_p(L,\delta)$ be the set of vertices contained in the parallelogram with boundary vertices
\begin{align}
u_0 = (-1.5\delta L,0) & \qquad u_1 = ((1+1.5\delta)L, (1+3\delta)L/\alpha(p)) \nonumber \\
v_0 = (-0.5\delta L,0) & \qquad v_1 = ((1+2.5\delta)L, (1+3\delta L)/\alpha(p)) \label{eq:sigma}
\end{align}
We say that $\sigma_p(L,\delta)$  contains a \emph{centered crossing} if there is a $p$-open path contained in $\sigma_p(L,\delta)$ that starts and ends near the middle of the boundary. More precisely, the crossing must start at a vertex in $[\lceil -1.25\delta L, -.75\delta L\rceil] \times \{0\}$ and end at a vertex in $[\lceil (1+1.75 \delta)L, (1+ 2.25 \delta L)\rceil ] \times \{\lceil (1 + 3 \delta) L / \alpha(p)\rceil \}$. Given $0<\epsilon<3^{-36}$, let $L(p,\epsilon)$ be the minimum value such that the probability of a centered crossing is at least $1-\epsilon$: 
$$L(p,\epsilon) = \min \{ L \colon P(\text{there exists a centered crossing of $\sigma_p(L,1/10)$}) > 1- \epsilon
\}.$$
We assume that $\epsilon <3^{-36}$ to ensure that $L(p,\epsilon) \to \infty$ as $p \downarrow p_c$. The $3^{-36}$ comes from a contour argument in \cite[Section 10]{RD84} which implies, if $L(p, \epsilon)$ increased to a finite limit, then we would have a block construction that guaranteed survival at the critical value, and by continuity of finite block probabilities also at a $p < p_c$. We assume that the following correlation exponent $\npa$ exists and is near its conjectured value: 
\begin{align}\lim_{ p \downarrow p_c} \f{\log L(p,\epsilon)}{\log  (p- p_c)} = \npa .\label{eq:npa}	
\end{align}

This is how long the parallelogram must be to ensure there is a centered crossing as $p$ approaches $p_c$. We stress that $\npa$ is widely believed to exist, but its existence has not been rigorously established for oriented or unoriented percolation. 



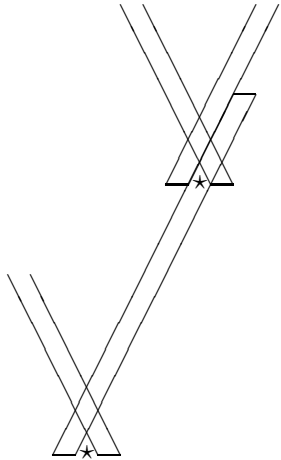
\begin{figure}[h]
\begin{center}
\setlength{\unitlength}{0.30mm}
\begin{picture}(140,220)
\put(30,30){\line(1,2){80}}
\put(40,30){\line(1,2){80}}
\put(30,30){\line(1,0){10}}
\put(110,190){\line(1,0){10}}
\put(50,30){\line(-1,2){40}}
\put(60,30){\line(-1,2){40}}
\put(50,30){\line(1,0){10}}
\put(42,28){$\star$}
\put(92,148){$\star$}
\put(100,150){\line(1,0){10}}
\put(100,150){\line(-1,2){40}}
\put(110,150){\line(-1,2){40}}
\put(80,150){\line(1,0){10}}
\put(80,150){\line(1,2){40}}
\put(90,150){\line(1,2){40}}
\end{picture}
\vspace{-1 cm}
\caption{Picture of the block construction. Stars mark points of the renormalized lattice. When a parallelogram has an open path we declare the corresponding edge open. }
\label{fig:block}
\end{center}
\end{figure}

\begin{proof}[Proof of \thref{thm:OZ}]
To ensure that our parallelograms are not too long and that our block construction will percolate, we fix $\lambda < 1/ \npa$ and $0<\epsilon <3^{-36}$.  Set $p_n = p_c + 2^{-\lambda n}$ and let $L_n = L(p_n,\epsilon)$. The definition of $\npa$ ensures that
\begin{align}L_n =O(( p - p_c )^{\npa})= O( 2^{\lambda \npa n}) = o(2^n) \label{eq:Ln}.\end{align}
 Choose $n_1$ large enough so that $L_n < 2^{n-1}$ for all $n \geq n_1$. 
 
 As in \cite{OPperc}, we divide the space into strips $\mathbb Z \times [2^m, 2^{m+1}]$. In each strip we lower bound the growth by considering percolation with parameter $p_{m+1}$ and do a block construction as described in \cite[Section 9]{RD84}. When we get to the top and the parallelograms do not completely fit in the strip, we regard them as part of the next strip.
%
%
By identifying each parallelogram to an edge in a renormalized lattice, we obtain a $1$-dependent version of oriented bond percolation on a new copy of $\mathcal L$ in which edges are open with probability $1-\epsilon$. Our choice of $\epsilon$ ensures that this 1-dependent process almost surely contains an infinite component because of the uniform bound $p_c < 1- \epsilon$ for all 1-dependent percolation models proven in \cite[Section 10]{RD84}.

 The passage time of a $(p_c+ 2^{-\lambda n})$-open crossing from height $2^n$ to $2^{n+1}$ is bounded by the number of edges with $\omega_e \in (p_c, p_n]$. Since there are $2^n$ edges and the $\omega_e$ are uniform random variables, this has expected value that is $$O((p_n - p_c) 2^n) = O(2^{ -\lambda n} 2^{n}).$$ Worst case, the time to cross each of these edges is  $F^{-1}( p_n) = O(2^{-(\lambda/a) n })$. This bound on $F^{-1}(p_n)$ comes from our hypothesis that $F(x) = p_c + O(x^a)$. Thus, conditional on there being a crossing, the expected passage time from height $2^n$ to height $2^{n+1}$ is  $$O(2^{ -\lambda n} 2^{n} 2^{-(\lambda /a) n}) = O(2^{[1-\lambda (  1 +1/a)]n}).$$ 
 
If $1- \lambda(1+ 1/a) < 0$, then the exponent is negative, and thus the total expected passage time along a connected chain of parallelograms is summable. Choosing $\lambda$ near $1/ \npa$, the condition on $a$ becomes the extra hypothesis in \thref{thm:OZ}. Recall that the origin in the renormalized lattice corresponds to the height $2^{n_1}$, and the infinite connected chain begins at some almost surely finite height above this.  Since we assume that $F(+\infty) =1$, the time to reach this infinite connected chain is almost surely finite. It follows that $\rho(\OZ^2,F)< \infty$ almost surely. 
\end{proof}

\subsection*{Acknowledgements}
Thanks to Rick Durrett for helpful feedback and discussion, as well as to Michael Damron for lively and useful discussions about critical first passage percolation.

\bibliographystyle{alpha}
\bibliography{frog_explostion.bib}
	
\end{document}